\pdfoutput=1
\documentclass{amsart}
\usepackage[longtable]{multirow}
\usepackage{graphicx}
\usepackage{longtable}
\usepackage{amsmath}
\usepackage{amsthm}
\usepackage{amsfonts}
\usepackage{mathtools}
\usepackage{fullpage}
\usepackage{mathtools}
\usepackage{tikz}
\usepackage{amssymb}
\newtheorem{theorem}{Theorem}
\newtheorem*{theorem*}{Theorem}

\newtheorem{lemma}{Lemma}
\newtheorem{corollary}{Corollary}

\theoremstyle{definition}
\newtheorem*{definition}{Definition}

\usepackage{rotating}
\usepackage{accents}
\usepackage[margin=1in]{geometry}

\title{Shifting local exponents of Picard-Fuchs operators}
\author[]{Tymoteusz Chmiel}

\begin{document}

\maketitle

\begin{abstract}
We investigate the operation of shifting local exponents and study its effects on the monodromy representation of a one-parameter family of Calabi-Yau threefolds. The main result is a characterization of shifts of geometric operators which are also geometric. We use this description to construct some Picard-Fuchs operators with interesting properties.
\end{abstract}

\section*{Introduction}

Picard-Fuchs operators of one-parameter families of Calabi-Yau threefolds have long been the subject of intensive research from many different perspectives. A fact crucial in all applications is that Picard-Fuchs operators are Fuchsian, i.e. they have only regular singular points.

A simple way of constructing new Fuchsian operators from a given one is the \mbox{\emph{shift of local exponents}}. Assume that $y(t)$ satisfies a Fuchsian equation $\mathcal{P}y=0$ and $f(t)$ is a function on the parameter space. We say that an operator $\mathcal{Q}$ is \emph{\mbox{a shift of $\mathcal{P}$,}} if the product $z(t):=f(t)y(t)$ satisfies the equation $\mathcal{Q}z=0$, 

Not every shift of a Picard-Fuchs operator is itself of geometric origin. For an operator coming from a one-parameter family of Calabi-Yau threefolds, the cup product induces a non-degenerate alternating form on the space of solutions and the monodromy group is conjugate to a subgroup of the \mbox{symplectic group $\mathrm{Sp}(4,\mathbb{Z})$.} This is no longer true for arbitrary shifts.

The goal of this paper is to study the effects of shifting local exponents on the symplectic structure of a Fuchsian operator. We give a complete characterization of shifts for which both $\mathcal{P}$ and $\mathcal{Q}$ admit a monodromy-invariant symplectic form and conclude that only very special shifts of Picard-Fuchs operators can themselves be Picard-Fuchs operators. We give two applications of this general result.

Firstly, we construct a Fuchsian operator with local characteristics of a Picard-Fuchs operator: its local monodromy operators are quasi-unipotent and symplectic. However, this operator does not admit a global monodromy-invariant symplectic form. Thus it cannot be a Picard-Fuchs operator of any family of Calabi-Yau threefolds. The operator is a shift of the Picard-Fuchs operator of certain family of double octics.

The second application focuses on the index of the monodromy group in $\mathrm{Sp}(4,\mathbb{Z})$. Whether this index is finite is known only in a handful of cases. Recently, a condition sufficient for the infiniteness of the index was given in \cite{simion}. All operators known to have monodromy group of infinite index satisfy this condition and it is natural to ask if under some assumptions this condition is also necessary. Using shifts of local exponents, we show that any operator satisfying the condition can be transformed into an operator for which the condition fails, yet whose monodromy group is still of infinite index. For Picard-Fuchs operators of double octics this shift can be performed on the level of the family. This gives many examples of Picard-Fuchs operators with monodromy group of infinite index which do not satisfy the sufficient condition of \cite{simion}.

The structure of the paper is as follows. Section \ref{sec:Picard-Fuchs} covers basic theory of Picard-Fuchs operators and introduces shifts of local exponents. In Section \ref{sec:sp-shifts} we prove our main results on shifts of operators admitting a monodromy-invariant symplectic form. Section \ref{sec:app} contains applications of these results to Picard-Fuchs operators of double octics.

\section{Preliminaries: Picard-Fuchs operators}\label{sec:Picard-Fuchs}

\subsection{One-parameter families of Calabi-Yau threefolds}\label{ssec:PFfamilies}

A \emph{Calabi-Yau threefold} is a smooth complex projective variety $X$ of dimension $3$ satisfying
$\Omega^{3}_X\simeq\mathcal{O}_X$ and $H^1(X,\mathcal{O}_X)=0$. By the Bogomolov-Tian-Todorov theorem there exists a universal family $\mathcal{X}\rightarrow S$, $b\in S$, where $S$ is a complex manifold with $\dim S=h^{2,1}(X)$ and the fiber $X_{b}$ is isomorphic to $X$. For every $t\in S$ the fiber $X_t$ is a Calabi-Yau threefold with $h^{2,1}(X_t)=h^{2,1}(X)$. When $h^{2,1}(X_t)=1$, we say that $\mathcal{X}\rightarrow S$ is a \emph{one-parameter family}.

Fix a holomorphic family $\omega_t\in H^3(X_t,\mathbb{C})$ and a smooth family $\gamma_t\in H_3(X_t,\mathbb{C})$, defined in some neighbourhood $U$ of the \mbox{base point $b$}. A \emph{period function} associated with these choices is $y(t):=\int_{\gamma_t}\omega_t$. It satisfies a linear differential equation $\mathcal{P}y=0$ for some differential operator
\begin{equation}\label{eq:d-form}
\mathcal{P}=q_n(t)\frac{d^n}{dt^n}+q_{n-1}(t)\frac{d^{n-1}}{dt^{n-1}}+\cdots+q_2(t)\frac{d^2}{dt^2}+q_1(t)\frac{d}{dt}+q_0(t)\quad\quad q_i(t)\in\mathcal{O}(U)\textnormal{ for } i=0,\cdots,n	
\end{equation}
\noindent The operator $\mathcal{P}$ is called a \emph{Picard-Fuchs operator} of the family $\mathcal{X}$. The space of solutions of $\mathcal{P}=0$ near $t\in U$ is isomorphic to $H^3(X_t,\mathbb{C})$. This identification defines a variation of Hodge structures on the local system of solutions.

In this paper we study one-parameter families over (Zariski-open subsets of) the projective line.
\begin{definition}\label{def}
We call a differential operator \emph{geometric} if it is a Picard-Fuchs operator of a family of Calabi-Yau threefolds $\mathcal{X}\rightarrow\mathbb{P}^1\setminus\Sigma$, where $\Sigma$ is a finite set of singular points and $h^{2,1}(X_t)=1$ for $t\not\in\Sigma$.
\end{definition}
\noindent If $\mathcal{P}$ is a geometric operator, we may assume that the coefficients $q_i(t)$ in (\ref{eq:d-form}) are rational functions (see \cite{Hofmann}). Furthermore, we have $n=b_3(X_t)=4$.

Put $S:=\mathbb{P}^1\setminus\Sigma$. Let $V$ be the space of solutions of $\mathcal{P}=0$ near $b\in S$. Any solution $y\in V$ can be continued analytically along any loop $\gamma:[0,1]\rightarrow S$, $\gamma(0)=\gamma(1)=b$, and the result $M_{\gamma}(y)$ is still a solution of $\mathcal{P}=0$. This defines a representation $M:\pi_1(S,b)\ni[\gamma]\mapsto M_\gamma\in GL(V)$, called \textit{the monodromy representation}. The image of the monodromy representation is \emph{the monodromy group} \mbox{$Mon(\mathcal{P})\subset\mathrm{GL}(V)\simeq\mathrm{GL}(4,\mathbb{C})$.} If $\gamma$ encircles one singular point $s\in\Sigma$ once counterclockwise, we call $M_s:=M_\gamma\in\mathrm{GL}(V)$ \mbox{a \emph{local monodromy operator at $s$}.} The monodromy group is generated by $M_s$ with $s\in\Sigma$.

\subsection{Local exponents of geometric operators}\label{ssec:PFexponents}

An important property of Picard-Fuchs operators is that they are Fuchsian, i.e. they have only regular singular points. For a geometric operator $\mathcal{P}$ given as in (\ref{eq:d-form}), \mbox{$0$ is a regular} singular point if and only if for every $i=0,\cdots,n$ the rational function $\tfrac{q_i(t)}{q_n(t)}$ has at $t=0$ pole of order at most $n-i$.

Let $\Theta:=t\cdot\frac{d}{dt}$ denote the logarithmic derivative and consider a differential operator $\mathcal{P}$ in its $\Theta$-form:
\begin{equation}\label{eq:theta-form}
\mathcal{P}=\mathcal{P}(\Theta,t)=\Theta^n+p_{n-1}(t)\Theta^{n-1}+\cdots+ p_{1}(t)\Theta+p_0(t)
\end{equation}
Assume that $0$ is a regular singular point of $\mathcal{P}=0$. Then the coefficients $p_i(t)$, $i=0,\cdots,n-1$, are holomorphic in its neighbourhood. In particular, $p_i(0)$ is well-defined.

The \emph{local exponents} of $\mathcal{P}$ at $0$ are the complex roots of the indicial equation
$$
\mathcal{P}(X,0)=X^n+p_{n-1}(0)X^{n-1}+\cdots+ p_{1}(0)X+p_0(0)=0
$$
Local exponents at points $t_0\neq 0$ are computed using a change of variables. It is easy to check that a \mbox{point $t_0$} is non-singular if and only if the local exponents at $t_0$ are $0,1,\cdots,n-2,n-1$.

Assume that $\mathcal{P}$ is a geometric operator. The Frobenius method of solving Fuchsian equations gives a clear connection between local exponents at the origin and the space of solutions in its neighbourhood. \mbox{Let $\alpha_1,\cdots,\alpha_k$,} $1\leq k\leq 4$, be the local exponents at $0$ written without repetitions and \mbox{let $m_i$,} \mbox{$i=1,\cdots,k$,} be their respective multiplicities. Write $t^\alpha:=e^{\alpha\log t}$.

Fix $i\in\{1,\cdots,k\}$. Let $\alpha:=\alpha_{i}$ and $m:=m_{i}$. Then there exist holomorphic functions $g_1,\cdots,g_{m}$, non-vanishing at $0$, such that the (possibly multi-valued) functions
$$\begin{array}{l}
y_1(t):=t^{\alpha}\cdot g_1(t)\\
y_2(t):=t^{\alpha}\cdot\big(g_2(t)+\log(t)g_1(t)\big)\\
\quad\cdots\quad \\
y_{m}(t):=t^{\alpha}\cdot\left(\displaystyle\sum_{l=0}^{m-1}\binom{m-1}{l}\log^l(t)g_{m-l}(t)\right)
\end{array}$$
form a linearly independent set of solutions of $\mathcal{P}=0$ in a neighbourhood of the origin.

Union of these sets forms a basis of solutions, called \mbox{the \emph{Frobenius basis}.} In particular, the number of Jordan blocks of the local monodromy operator $M_0$ is equal to the number of distinct local exponents and the eigenvalues of $M_0$ are $e^{2\pi i\alpha_1},\cdots,e^{2\pi i\alpha_k}$. By the Local Monodromy Theorem, $M_0$ is quasi-unipotent: there exists $r\in\mathbb{N}_{>0}$ such that \mbox{$\left(M_0^r-\operatorname{Id}\right)^4=0$}. It follows that all local exponents are rational.

\subsection{Symplectic structure of geometric operators}\label{ssec:PFsymplectic}

Another important property of geometric operators is their symplectic structure. Let $V$ be the space of solutions of a geometric equation $\mathcal{P}=0$ near a non-singular point $b$. Recall that $V$ is isomorphic to $H^3(X_{b},\mathbb{C})$. In particular, it contains a monodromy-invariant lattice $H^3(X_{b},\mathbb{Z})$. The cup product defines a non-degenerate alternating form on $V$ which is also monodromy-invariant. Thus there exists a \mbox{\textit{symplectic basis of solutions},} giving an identification $V\simeq\mathbb{C}^4$ such that $Mon(\mathcal{P})\subset \mathrm{Sp}(4,\mathbb{Z})$.

Given a subgroup $G\subset\mathrm{Sp}(4,\mathbb{Z})$ a natural question is whether the index $[ \mathrm{Sp}(4,\mathbb{Z}):G]$ is finite. If that is the case, $G$ is called \emph{arithmetic}. Clearly arithmetic subgroups are Zariski-dense. If $G$ is Zariski-dense but of infinite index, it is called \emph{thin}.

The monodromy group $Mon(\mathcal{P})\subset\mathrm{Sp}(4,\mathbb{Z})$ can be arithmetic or thin, and at present there is no general procedure for determining whether its index is finite. Even in the simplest case of hypergeometric operators it took considerable effort to establish which monodromy groups are arithmetic (see \cite{Brav-Thomas,Singh,Singh-Venkataramana}). For other operators one has only sporadic results \mbox{(see \cite{Hofmann-van Straten,orphan2}).}

Recently, two sufficient conditions for the \emph{infiniteness} of the index were given in \cite{simion}. The conditions, called \emph{Assumptions A and B}, are defined in terms of local Kodaira-Spencer maps (\cite[Definitions 2.1.9 and 5.4.1]{simion}). If they are satisfied at every point, the monodromy representation is log-Anosov \mbox{(\cite[Definition 4.3.2]{simion})}. This implies that the monodromy group is of infinite index. The following lemma shows how one can verify Assumptions A and B using local exponents.

\begin{lemma}{\cite[Proposition 3.1.8]{simion}}\label{lem:AB}
	Let $\alpha_1\leq\alpha_2\leq\alpha_3\leq\alpha_4$ be local exponents of a geometric operator $\mathcal{P}$ at a point $s\in\mathbb{P}^1$. Put $k:=\#\{\alpha_1,\cdots,\alpha_4\}$. If $k=4$, let $N\in\mathbb{N}_{>0}$ be the order of local monodromy at $s$.
	\begin{itemize}
		\item If $k=1$, the point $s$ satisfies Assumption A and Assumption B.
		\item If $k=2$, the point $s$ satisfies Assumption B but not Assumption A.
		\item If $k=3$, the point $s$ satisfies Assumption A but not Assumption B.
		\item If $k=4$, the point $s$ satisfies Assumption B if and only if $N(\alpha_2-\alpha_1)=1$.
		\item If $k=4$, the point $s$ satisfies Assumption A if and only if $N(\alpha_3-\alpha_2)=1$.
	\end{itemize}
	
	If every point satisfies Assumption A, then $Mon(\mathcal{P})\subset\mathrm{Sp}(4,\mathbb{Z})$ is of infinite index.
	
	If every point satisfies Assumption B, then $Mon(\mathcal{P})\subset\mathrm{Sp}(4,\mathbb{Z})$ is of infinite index.
\end{lemma}
In the Introduction of \cite{simion}, the author notes that it would be interesting to understand when the disjunction of Assumptions A and B is also necessary for the infiniteness of the index. For example, a hypergeometric operator $\mathcal{P}$ satisfies Assumption A if and only if the monodromy group $Mon(\mathcal{P})$ is thin. We come back to this question in Section \ref{ssec:A/B}.

\subsection{Shift of local exponents}\label{ssec:shifts}

	In the next two sections we study an operation on the set of Fuchsian operators, which we call the \emph{shift of local exponents}. Let us start with the definition and then briefly discuss the motivation.
	
	\begin{definition}
		Let $\mathcal{P}=\mathcal{P}(\Theta,t)$ be a differential operator and let $\alpha\in\mathbb{C}$. The operator $\mathcal{Q}$ is obtained from $\mathcal{P}$ via the \emph{shift of local exponents at $0$ by $\alpha$} when $\mathcal{Q}=\mathcal{Q}(\Theta,t)=\mathcal{P}(\Theta-\alpha,t)=:\mathcal{P}(\alpha)$.
	\end{definition}
	
	Let $y(t)$ be a local solution of the Picard-Fuchs equation $\mathcal{P}=0$, associated with a family of $3$-forms \mbox{$\omega_t\in H^3(X_t,\mathbb{C})$}. For any holomorphic function $f(t)$, $f(t)\omega_t$ also defines a holomorphic family of $3$-forms for the \mbox{family $X_t$.} The period function associated with this family is \mbox{$z(t):=f(t)y(t)$} and it satisfies $\mathcal{Q}z=0$ for some Fuchsian operator $\mathcal{Q}$. We say that $\mathcal{Q}$ is the \mbox{\emph{shift of $\mathcal{P}$ by f(t)}.}
	
	Assume $\mathcal{P}$ is geometric and $f(t)$ is	an algebraic function. The function $f(t)$ defines a meromorphic function on an appropriate Riemann surface and $z(t)$ is a period function of the pull-back of the family $X_t$ by the covering map. Thus usually the shifted operator $\mathcal{Q}$ is defined over a curve of genus $g>0$.
	
	For this reason we focus on the case when some power of $f(t)$ is meromorphic. Then the \mbox{shift of $\mathcal{P}$ by $f(t)$} can be obtained by repeated applications of M\"obius transformations and shifts by $t^\alpha$ for some $\alpha\in\mathbb{Q}$. The following lemma shows that the shift of $\mathcal{P}$ by $t^\alpha$ is precisely $\mathcal{P}(\alpha)$.
	
	\begin{lemma}\label{lem:shift}
		Let $\mathcal{P}$ be a Fuchsian operator and let $\alpha_1,\cdots,\alpha_n$ be the local exponents of $\mathcal{P}$ at $0$. Fix $\alpha\in\mathbb{C}$.
		
		The shifted operator $\mathcal{P}(\alpha)$ is Fuchsian and the local exponents of $\mathcal{P}(\alpha)$ at $0$ are $\alpha_1+\alpha,\cdots,\alpha_n+\alpha$.
		
		Furthermore, \mbox{if $y(t)$ is} a solution of $\mathcal{P}=0$, then $t^{\alpha}y(t)$ is a solution of $\mathcal{P}(\alpha)=0$.
	\end{lemma} 
	
	\begin{proof}
		If the operator $\mathcal{P}$ is given as in (\ref{eq:theta-form}), then $$\mathcal{P}(\alpha)=
		\mathcal{P}(\Theta-\alpha,t)=(\Theta-\alpha)^n+p_{n-1}(t)(\Theta-\alpha)^{n-1}+\cdots+ p_{1}(t)(\Theta-\alpha)+p_0(t)
		$$
		It immediately follows that $\mathcal{P}(\alpha)$ is Fuchsian and the local exponents have the desired form.
		
		Since $\Theta t^\alpha=\alpha t^\alpha$, the Leibniz rule gives $\left(\Theta-\alpha\right)t^\alpha f(t)=t^\alpha \Theta f(t)$ for any differentiable function $f(t)$. Consequently, $\left(\Theta-\alpha\right)^k t^\alpha y(t)=t^\alpha \Theta^k y(t)$ and 
		\mbox{$\mathcal{P}(\alpha)t^\alpha y(t)=\mathcal{P}(\Theta-\alpha,t)t^\alpha y(t)=t^{\alpha}\cdot \mathcal{P}(\Theta,t)y(t)=0$.}
	\end{proof}
	
	A shift $\mathcal{P}(\alpha)$ of a geometric operator $\mathcal{P}$ is not necessarily geometric. For example if $\alpha\not\in\mathbb{Q}$, the local monodromy of $\mathcal{P}(\alpha)$ at $0$ is not quasi-unipotent, hence not geometric by the Local Monodromy Theorem. In the next section we examine how shifting local exponents affects the existence of a symplectic structure.

\section{Shifts of symplectic operators}\label{sec:sp-shifts}

The standard symplectic form, which is invariant under the action of $\mathrm{Sp}(4,\mathbb{C})$, is given by the matrix
$$\begin{pmatrix}
	0 & 1 & 0 & 0\\
	-1 & 0 & 0 & 0\\
	0 & 0 & 0 & 1\\
	0 & 0 & -1 & 0
\end{pmatrix}$$
A matrix is \emph{symplectic} if it is conjugate to an element of $\mathrm{Sp}(4,\mathbb{Z})$. We call a Fuchsian operator \emph{symplectic} if its monodromy group is conjugate to a Zariski-dense subgroup of $\mathrm{Sp}(4,\mathbb{Z})$. We make free and implicit use of the fact that $\mathrm{Sp}(4,\mathbb{Z})$ is Zariski-dense in $\mathrm{Sp}(4,\mathbb{Q})$ and $\mathrm{Sp}(4,\mathbb{Q})$ is Zariski-dense in $\mathrm{Sp}(4,\mathbb{C})$.

\begin{lemma}\label{lem:1/4}
Let $\alpha\in\mathbb{C}$. If local monodromy at $0$ of both $\mathcal{P}$ and $\mathcal{P}(\alpha)$ is symplectic, then $4\alpha\in\mathbb{Z}$.
\end{lemma}	

\begin{proof}
Let $M_0$ be a matrix of local monodromy of $\mathcal{P}$ at $0$. By Lemma \ref{lem:shift} a matrix of local monodromy of $\mathcal{P}(\alpha)$ at $0$ is $e^{2\pi i\alpha}M_0$. Since $\mathrm{Sp}(4,\mathbb{C})\subset\mathrm{SL}(4,\mathbb{C})$, we must have
$1=\det(e^{2\pi i\alpha}M_0)=e^{8\pi i\alpha}\det(M_0)=e^{8\pi i\alpha}.$
\end{proof}

\begin{lemma}\label{lem:i}
Let $G_1=\langle M_0,M_1\cdots,M_n\rangle\subset\mathrm{Sp}(4,\mathbb{Q})$ be a finitely generated, Zariski-dense subgroup and put \mbox{$G_2:=\langle i\cdot M_0,M_1\cdots,M_n\rangle\subset\mathrm{GL}(4,\mathbb{C})$}. Then every $G_2$-invariant alternating form on $\mathbb{C}^4$ is degenerate.
\end{lemma}	

\begin{proof}
Let $\Omega_1$ be the matrix of the standard symplectic form and let $\Omega_2$ be the matrix of a $G_2$-invariant alternating form. Let $\Omega=(a_{i,j})_{i,j=1,\cdots,4}$, $a_{i,j}=-a_{j,i}$, be a generic skew-symmetric matrix of variables. Then $\Omega_1$ is a solution of the linear system of equations $M_i^T\Omega M_i=\Omega$, \mbox{$i=0,\cdots,n$,} while $\Omega_2$ is a solution of the system $M_i^T\Omega M_i=\Omega$ for $i=1,\cdots,n$ and $\Omega=(i\cdot M_0)^T\Omega(i\cdot M_0)=-M_0^T\Omega M_0$. The matrices $M_i$ have rational coefficients, thus rational solutions of these systems are dense in the space of all complex solutions. It follows that if there is a non-degenerate solution of the second system, there is a non-degenerate rational solution. Assume that $\Omega_2\in\mathrm{GL}(4,\mathbb{Q})$ and denote the entries of $\Omega_2$ as follows:
$$
\Omega_2=\begin{pmatrix}
0 & b & c & d\\
-b & 0 & g & h\\
-c & -g & 0 & l\\
-d & -h & -l & 0
\end{pmatrix}\quad\textnormal{with}\quad b,c,d,g,h,l\in\mathbb{Q},\quad\det\Omega_2=(-bl + ch - dg)^2\neq 0
$$

Let $\Omega_t:=\Omega_1+it\cdot\Omega_2$. For any $q\in\mathbb{Q}$ we have $\overline{\Omega_{q}}=\overline{\Omega_1+iq\cdot\Omega_2}=\overline{\Omega_1}+\overline{iq\cdot\Omega_2}=\Omega_1-iq\cdot\Omega_2=\Omega_1+iq\cdot(-\Omega_2)=M_0^T\Omega_1 M_0+iq M_0^T\Omega_2 M_0=M_0^T\Omega_{q}M_0$, where $\overline{(\ \  )}$ denotes the complex conjugation. It follows that $\overline{\det\Omega_{q}}=\det\overline{\Omega_{q}}=\det(M_0^T\Omega_{q}M_0)=\det M_0^2\det\Omega_{q}=\det\Omega_{q}$, because $M_0$ is symplectic. Thus $\det\Omega_t$, considered as a polynomial in $t$, has rational coefficients. On the other hand, we have
\vspace{-1mm}
\begin{equation*}
\det\Omega_t=(-bl + ch - dg)^2t^4 + 2i(b + l)(-bl + ch - dg)t^3 + (-2bl + 2ch - 2gd + i(b + l)^2)t^2 + i(2b + 2l)t+1
\vspace{-1mm}
\end{equation*}
Since the parameters are rational, the coefficients of $t$ and $t^3$ must vanish, i.e. $b=-l$.
 
We may further reduce the number of parameters by a symplectic change of basis, i.e. replacing $\Omega_2$ by $Q^T\Omega_2 Q$ for some $Q\in\mathrm{Sp}(4,\mathbb{Q})$. For this we assume $c\neq 0$. The proof in the case $c=0$ is essentially the same, differing only in the details of computations.

Put
$$
Q:=\begin{pmatrix}
	1 & -\frac{g}{c} & 0 & \frac{l}{c}\\
	\frac{g}{c} & 1 & 0 & 0\\
	0 & \frac{l}{c} & 1 & -\frac{d}{c}\\
	0 & 0 & 0 & 1
\end{pmatrix}\in\mathrm{Sp}(4,\mathbb{Q})
$$
and $a:=\frac{l^2+ch-dg}{c}$. We redefine $G_i:=Q^{-1}G_iQ$ for $i=1,2$, $M_i:=Q^{-1}M_iQ$ for $i=0,\cdots,n$,
$$
\Omega_2:=Q^T\Omega_2 Q=\begin{pmatrix}
	0 & 0 & c & 0\\
	0 & 0 & 0 & a\\
	-c & 0 & 0 & 0\\
	0 & -a & 0 & 0
\end{pmatrix}\quad\textnormal{with}\quad \det\Omega_2=(ac)^2\neq 0
$$
and $\Omega_t:=\Omega_1+t\cdot\Omega_2$ with $\det\Omega_t=(at^2-1)^2$. Let $t_1=\frac{1}{\sqrt{a}},t_2=-\frac{1}{\sqrt{a}}$ be the roots of $\det\Omega_t=0$.

Let $G_3:=\langle M_1,\cdots,M_n\rangle\subset G_1\cap G_2\subset\mathrm{Sp}(4,\mathbb{Q})$. Both $\Omega_1$ and $\Omega_2$ are $G_3$-invariant, thus all forms $\Omega_t$ are $G_3$-invariant as well. In particular, the radical $R_t:=\{v\in\mathbb{C}^4: \Omega_tv=0\}$ is a $G_3$-invariant subspace. For $t=t_1,t_2$ this gives two non-trivial $G_3$-invariant subspaces
$$
R_{t_1}=\operatorname{span}\left\{
e_1:=\begin{pmatrix}
\sqrt{a}\\
0\\
0\\
c
\end{pmatrix},\ 
f_1:=\begin{pmatrix}
0\\
c\\
-\sqrt{a}\\
0	
\end{pmatrix}
\right\},\quad\quad
R_{t_2}=\operatorname{span}\left\{
e_2:=\begin{pmatrix}
-\sqrt{a}\\
0\\
0\\
c
\end{pmatrix},\ 
f_2:=\begin{pmatrix}
0\\
c\\
\sqrt{a}\\
0	
\end{pmatrix}
\right\}
$$

Clearly $e_1,f_1,e_2,f_2$ are linearly independent. Let $N:=[e_1,f_1,e_2,f_2]\in\mathrm{GL}(4,\mathbb{C})$. A direct computation shows that $N^T\Omega_1 N=2\sqrt{a}c\cdot\Omega_1$, so $N$ lies in the normalizer of $\mathrm{Sp}(4,\mathbb{C})$. Thus we may change the basis \mbox{via $N$} and still have $G_1:=N^{-1}G_1N\subset\mathrm{Sp}(4,\mathbb{C})$. After this change of basis the matrix
$$\Omega_2:=\frac{1}{-2ac}N^T\Omega_2N=\begin{pmatrix}
0 & 1 & 0 & 0\\
-1 & 0 & 0 & 0\\
0 & 0 & 0 & -1\\
0 & 0 & 1 & 0
\end{pmatrix}$$
represents an alternating form invariant under the action of $G_2:=N^{-1}G_2N$. Furthermore, the elements of $G_3:=N^{-1}G_3N$ all have the block form
\begin{equation}\label{eq:block-form}
\begin{pmatrix}
	A & 0_{2,2} \\
	0_{2,2} & B \\
\end{pmatrix}\quad\textnormal{for some}\quad A,B\in\mathrm{GL}(2,\mathbb{C})
\end{equation}
because the subspaces generated by $e_1,f_1$ and $e_2,f_2$ were $G_3$-invariant.

Put $M_i:=N^{-1}M_iN$ for $i=0,\cdots,n$. Recall that we have $M_0^T\Omega_1M_0=\Omega_1$ and $M_0^T\Omega_2M_0=-\Omega_2$. Since the matrices $\Omega_1$ and $\Omega_2$ are now explicitly given, we can solve this system of equations treating the coefficients of $M_0=\left(m_{i,j}\right)_{i,j=1,\cdots,4}$ as unknown. Write
$$M_0^T\Omega_1M_0-\Omega_1=\left(x_{i,j}\right)_{i,j=1,\cdots,4}\quad\textnormal{ and }\quad M_0^T\Omega_2M_0+\Omega_2=\left(y_{i,j}\right)_{i,j=1,\cdots,4}$$
We have $x_{i,j}=y_{i,j}=0$ for all $i,j=1,\cdots,4$, so the following equalities must also hold:
$$
\begin{cases}
\frac{1}{2}(x_{1,2}+y_{1,2})=m_{1,1}m_{2,2} - m_{1,2}m_{2,1}=0\\
\frac{1}{2}(x_{1,3}+y_{1,3})=m_{1,1}m_{3,2} - m_{1,2}m_{3,1}=0\\
\frac{1}{2}(x_{1,4}+y_{1,4})=m_{1,1}m_{4,2} - m_{1,2}m_{4,1}=0\\
\frac{1}{2}(x_{2,3}+y_{2,3})=m_{2,1}m_{3,2} - m_{2,2}m_{3,1}=0\\
\frac{1}{2}(x_{2,4}+y_{2,4})=m_{2,1}m_{4,2} - m_{2,2}m_{4,1}=0\\
\frac{1}{2}(x_{3,4}+y_{3,4})+1=m_{3,1}m_{4,2} - m_{3,2}m_{4,1}=1\\
\end{cases}
\begin{cases}
	\frac{1}{2}(x_{3,4}-y_{3,4})=m_{3,3}m_{4,4} - m_{3,4}m_{4,3}=0\\
	\frac{1}{2}(x_{2,4}-y_{2,4})=m_{2,3}m_{4,4} - m_{2,4}m_{4,3}=0\\
	\frac{1}{2}(x_{2,3}-y_{2,3})=m_{2,3}m_{3,4} - m_{2,4}m_{3,3}=0\\
	\frac{1}{2}(x_{1,4}-y_{1,4})=m_{1,3}m_{4,4} - m_{1,4}m_{4,3}=0\\
	\frac{1}{2}(x_{1,3}-y_{1,3})=m_{1,3}m_{3,4} - m_{1,4}m_{3,3}=0\\
	\frac{1}{2}(x_{1,2}-y_{1,2})+1=m_{1,3}m_{2,4} - m_{1,4}m_{2,3}=1\\
\end{cases}
$$
Solving this system of equations we find that
$m_{1,1}=m_{2,2}=m_{1,2}=m_{2,1}=m_{3,3}=m_{4,4}=m_{3,4}=m_{4,3}=0$
and so $M_0$ must have the block form
\begin{equation}\label{eq:block-form0}
	\begin{pmatrix}
		0_{2,2} & A \\
		B & 0_{2,2} \\
	\end{pmatrix}\quad\textnormal{for some}\quad A,B\in\mathrm{GL}(2,\mathbb{C})
\end{equation}

If $X,Y\in\mathrm{GL}(4,\mathbb{C})$ are matrices of the form (\ref{eq:block-form}) or (\ref{eq:block-form0}), so are $X^{-1}$ and $XY$. It follows that every $M\in G_1$ has the block form (\ref{eq:block-form}) or (\ref{eq:block-form0}). But the set of such matrices is not \mbox{Zariski-dense in $\mathrm{Sp}(4,\mathbb{C})$.}
\end{proof}

Now we are ready to give a complete characterization of shifts preserving the symplectic structure.

\begin{theorem}\label{th:1/2}
Let $\mathcal{P}$ be a symplectic operator and let $\alpha\in\mathbb{C}$. Then $\mathcal{P}(\alpha)$ is symplectic if and only if $2\alpha\in\mathbb{Z}$.

Furthermore, if $2\alpha\in\mathbb{Z}$, then $Mon(\mathcal{P})$ is arithmetic if and only if $Mon\left(\mathcal{P}(\alpha)\right)$ is arithmetic.
\end{theorem}

\begin{proof}
Let $y_1,\cdots,y_4$ be a symplectic basis of solutions of $\mathcal{P}=0$ near $0$. Let $M_0$ be the matrix of local monodromy of $\mathcal{P}$ at $0$ and let $M_1,\cdots,M_n$ be the matrices of local monodromies at non-zero finite singularities, so that $G_1:=Mon(\mathcal{P})=\langle M_0,M_1,\cdots,M_n\rangle\subset\mathrm{Sp}(4,\mathbb{Z})$.

By Lemma \ref{lem:shift} functions $t^{\alpha}y_1,\cdots,t^{\alpha}y_4$ form a basis of solutions of $\mathcal{P}(\alpha)=0$ near $0$. In this basis local monodromy of $\mathcal{P}(\alpha)$ at $0$ is given by the matrix $e^{2\pi i\alpha}M_0$. Furthermore, matrices of local monodromies at other finite singularities are $M_1,\cdots,M_n$. Indeed, for any ray $R$ the function $t^\alpha$ is holomorphic on $\mathbb{P}^1\setminus R$. We may arrange so that $R$ contains no non-zero finite singularities and then compute their local monodromies using loops from \mbox{$\pi_1(\mathbb{P}^1\setminus\left( R\cup\Sigma\right),b)$.} Since holomorphic functions have trivial monodromy, it follows that the monodromy action on $t^{\alpha}y_1,\cdots,t^{\alpha}y_4$ is the same as the monodromy action on $y_1,\cdots,y_4$. Consequently $G_2:=Mon\left(\mathcal{P}(\alpha)\right)=\langle e^{2\pi i\alpha}M_0,M_1,\cdots,M_n\rangle\subset\mathrm{GL}(4,\mathbb{C})$.

Assume that $\mathcal{P}(\alpha)$ is symplectic. Then $4\alpha\in\mathbb{Z}$ by Lemma \ref{lem:1/4}. If $\alpha=\tfrac{2k+1}{4}$, $k\in\mathbb{Z}$, then
$e^{2\pi i\alpha}=e^{\tfrac{(2k+1)\pi i}{2}}=e^{k\pi i}e^{\tfrac{\pi i}{2}}=\pm i$
and $\mathcal{P}(\alpha)$ cannot be symplectic by Lemma \ref{lem:i}. Thus $2\alpha\in\mathbb{Z}$.

Now assume $2\alpha\in\mathbb{Z}$, so that $G_2=\langle e^{2\pi i\alpha}M_0,M_1,\cdots,M_n\rangle=\langle \pm M_0,M_1,\cdots,M_n\rangle\subset\mathrm{Sp}(4,\mathbb{Z})$. Since $G_1$ is Zariski-dense, so is $G_2$. Thus $P(\alpha)$ is symplectic. Put $G:=\langle-\operatorname{Id},M_0,M_1,\cdots,M_n\rangle\subset\mathrm{Sp}(4,\mathbb{Z})$. For $i=1,2$ we have $G_i\subset G$, $[G:G_i]\leq 2$ and $[\mathrm{Sp}(4,\mathbb{Z}):G_i]=[\mathrm{Sp}(4,\mathbb{Z}):G][G:G_i]$. It follows that $G_i$ is arithmetic if and only if $G$ is arithmetic.
\end{proof}

If $\mathcal{P}$ is a geometric operator, monodromy-invariance of the cup product implies that we can assume \mbox{$Mon(\mathcal{P})\subset\mathrm{Sp}(4,\mathbb{Z})$}. Thus we have the following corollary.
\begin{corollary}\label{cor:geometric-shift}
Let $\mathcal{P}$ be a geometric operator and fix $\alpha\in\mathbb{C}\setminus \frac{1}{2}\mathbb{Z}$. Assume that the monodromy group \mbox{$Mon(\mathcal{P})\subset\mathrm{Sp}(4,\mathbb{Z})$} is Zariski-dense. Then $\mathcal{P}(\alpha)$ is not geometric.
\end{corollary}

Examples of geometric operators whose monodromy group is not Zariski-dense can be found in \cite{non-dense1,non-dense2}. Their monodromy matrices have block-diagonal form as in (\ref{eq:block-form}). These were also the first examples of geometric operators without points of maximal unipotent monodromy \mbox{(MUM points),} \mbox{i.e. regular} singular points whose local monodromy operator has one Jordan block. Note that local monodromy at a MUM point can never be represented by a matrix as in (\ref{eq:block-form}) or (\ref{eq:block-form0}). It follows that in the proof of Theorem \ref{th:1/2} the assumption of Zariski-density can be replaced by the assumption that the operator has a MUM point.

\begin{corollary}\label{cor:mum}
Let $\mathcal{P}$ be a geometric operator and fix $\alpha\in\mathbb{C}\setminus \frac{1}{2}\mathbb{Z}$. Assume that $\mathcal{P}$ has a MUM singularity. Then $\mathcal{P}(\alpha)$ is not geometric.
\end{corollary}

Theorem \ref{th:1/2} can also be applied to more general shifts. Assume that $\mathcal{Q}$ is a shift of a symplectic \mbox{operator $\mathcal{P}$} by a function $f(t)$ such that $f(t)^n$ is rational for some $n\geq 1$. Then we can consider $\mathcal{Q}$ as obtained \mbox{from $\mathcal{P}$} by a sequence of M\"obius transformations and shifts of local exponents at $0$. M\"obius transformations do not affect the monodromy and shifts multiply appropriate generators by a scalar. If the subgroup of $Mon(\mathcal{P})$ generated by local monodromies around points not in the branching locus of $f(t)$ is Zariski-dense in $\mathrm{Sp}(4,\mathbb{Z})$, several applications of \mbox{Theorem \ref{th:1/2}} yield the following.

\begin{corollary}
	Let $\mathcal{P}$ be a symplectic operator, $f(t)$ a root of a rational function and \mbox{$\mathcal{Q}$ the shift of $\mathcal{P}$ by $f(t)$.}
	
	If $f(t)^2\in\mathbb{C}(t)$, then $\mathcal{Q}$ is symplectic and $Mon(\mathcal{P})$ is arithmetic if and only if $Mon(\mathcal{Q})$ is arithmetic.
	
	If $\mathcal{Q}$ is symplectic and $Mon(\mathcal{P})\cap Mon(\mathcal{Q})\subset \mathrm{Sp}(4,\mathbb{Z})$ is Zariski-dense, then $f(t)^2\in\mathbb{C}(t)$.
\end{corollary}

\section{Applications: double octic operators}\label{sec:app}

In this section we present two further applications of Theorem \ref{th:1/2}. First, we give an example of a Fuchsian operator which has local properties of a Picard-Fuchs operator near every point but is not a Picard-Fuchs operator of a one-parameter family. The second example is a geometric operator with thin monodromy group satisfying neither Assumption A, nor Assumption B.

The examples are constructed as shifts of Picard-Fuchs operators of one-parameter families of double octics. Among Calabi-Yau threefolds, double octics constitute a convenient source of examples. Their geometry is particularly simple and allows for explicit computations. On the other hand, this class is rich enough to reveal many unexpected phenomena (see \cite{non-liftable, orphans, CvS, orphan2, NOS}).

\subsection{Double octic Calabi-Yau threefolds}\label{ssec:double-octics}

Consider a double cover $X\xrightarrow{2:1}\mathbb{P}^3$ branched along an octic surface $O\subset\mathbb{P}^3$. If $O$ is smooth, $X$ is a Calabi-Yau threefold. We focus on the case when $O=P_1\cup\cdots\cup P_8\subset\mathbb{P}^3$ is a union of eight planes. The corresponding double cover $X$ is singular. Assume that the planes $P_1,\cdots,P_8$ are distinct and their union contains no $k$-fold lines for $k \geq 4$ and no $l$-fold points for $l \geq 6$. In this case singularities of $X$ can be resolved in such a way that the obtained resolution $\widetilde{X}$ is again a Calabi-Yau threefold (see \cite{double-covers}). We call $\widetilde{X}$ a \emph{double octic} defined by the \emph{octic arrangement} $O$.

Assume that the branching locus $O\subset\mathbb{P}^3$ is defined by an equation $P(x,y,z,w)=0$. If $O$ is an octic arrangement, $P$ is a product of eight linear factors. The double cover $X$ can be considered as a subvariety in the weighted projective space $\mathbb{P}(4,1,1,1,1)$ defined by the equation $u^2=P(x,y,z,w)$. The holomorphic $3$-form is given by $$\omega=\frac{dx\wedge dy\wedge dz}{u}=\frac{dx\wedge dy\wedge dz}{\sqrt{P(x,y,z,w)}}$$

Deformation of a double octic is a double octic, hence any one-parameter family of double octics $\widetilde{X}_t$ comes from a family of octic arrangements $O_t$ (see \cite{Cy-Ko}). If $P_t(x,y,z,w)=0$ is the equation defining $O_t$, the corresponding double cover $X_t$ is defined by $u^2=P_t(x,y,z,w)$. For a fixed family $\gamma_t\in H_3(\widetilde{X}_t,\mathbb{C})$ the period function is given by
\begin{equation}\label{period-function}
y(t)=\displaystyle\int_{\gamma_t}\frac{dx\wedge dy\wedge dz}{\sqrt{P_t(x,y,z,w)}}
\end{equation}
Picard-Fuchs operators of double octic families can be efficiently computed using the computer program described in \cite{lairez}.

\subsection{Locally geometric operators}\label{ssec:locall-not-globally}

We call a Fuchsian operator \emph{locally geometric} if local monodromy at every point is quasi-unipotent and symplectic. It follows from the Local Monodromy Theorem and the existence of cup product that geometric operators are locally geometric. Using \mbox{Theorem \ref{th:1/2}} we show that the converse does not hold.

For any geometric operator $\mathcal{P}$ and any $q\in\mathbb{Q}$ local monodromies of the shift $\mathcal{P}(q)$ are still quasi-unipotent. They are also symplectic at all singular points, except possibly at $0$ and $\infty$. If we find a \mbox{geometric operator $\mathcal{P}$} and $q\in\mathbb{Q}\setminus\frac{1}{2}\mathbb{Z}$ such that the local monodromies of $\mathcal{P}(q)$ at $0$ and $\infty$ are symplectic, the operator $\mathcal{P}(q)$ will be locally geometric. However, by Theorem \ref{th:1/2} we expect that in general it will not be geometric.

We begin by identifying conditions necessary for the existence of a shift with symplectic monodromy at the origin. Let $M_0$ be a matrix of the local monodromy at $0$ of some geometric operator $\mathcal{P}$. Since $M_0$ is symplectic, its possible Jordan forms are
$$
\begin{pmatrix}
	\zeta_1 & 1 & 0 & 0\\
	0 & \zeta_1 & 1 & 0\\
	0 & 0 & \zeta_1 & 1\\
	0 & 0 & 0 & \zeta_1
\end{pmatrix},\quad
\begin{pmatrix}
	\zeta_1 & 1 & 0 & 0\\
	0 & \zeta_1 & 0 & 0\\
	0 & 0 & \zeta_2 & 1\\
	0 & 0 & 0 & \zeta_2
\end{pmatrix},\quad
\begin{pmatrix}
	\zeta_1 & 1 & 0 & 0\\
	0 & \zeta_1 & 0 & 0\\
	0 & 0 & \zeta_2 & 0\\
	0 & 0 & 0 & \zeta_3
\end{pmatrix}\quad\textnormal{and}\quad
\begin{pmatrix}
	\zeta_1 & 0 & 0 & 0\\
	0 & \zeta_2 & 0 & 0\\
	0 & 0 & \zeta_3 & 0\\
	0 & 0 & 0 & \zeta_4
\end{pmatrix}
$$
for some $\zeta_j\in\mathbb{C}$, $j\in J\subset\{1,\cdots,4\}$. By the Local Monodromy Theorem eigenvalues $\zeta_j$ are roots of unity. For every $j\in J$ denote by $n_j\in\{1,2,4\}$ the size of the Jordan block corresponding to $\zeta_j$. Let $E$, resp. $E'$, be the set of pairs $\left\{(\zeta_j,n_j)\right\}_{j\in J}$, resp. $\left\{(i\zeta_j,n_j)\right\}_{j\in J}$. Then $E,E'\subset\mathbb{C}\times\mathbb{N}$ and the latter carries a natural action of the Galois group $\operatorname{Gal}(\mathbb{C}/\mathbb{Q})$, trivial on the second factor.

Assume that the local monodromy \mbox{at $0$} of $\mathcal{P}(q)$ is also symplectic. Then both $M_0$ and $e^{2\pi i q} M_0$ are conjugate to an element of $\mathrm{Sp}(4,\mathbb{Z})$. By Lemma \ref{lem:1/4} we have $4q\in\mathbb{Z}$. Using complex conjugation and an integral shift, we can arrange so that $q=\frac{1}{4}$. Thus $M_0$ and $i\cdot M_0$ are both conjugate to an integral matrix. It follows that the sets $E$ and $E'$ must be $\operatorname{Gal}(\mathbb{C}/\mathbb{Q})$-invariant.

A short case-by-case analysis shows that possible Jordan forms of $M_0$ are
$$
\begin{pmatrix}
	1 & 1 & 0 & 0\\
	0 & 1 & 0 & 0\\
	0 & 0 & -1 & 1\\
	0 & 0 & 0 & -1
\end{pmatrix}\textnormal{and}
\begin{pmatrix}
	i & 1 & 0 & 0\\
	0 & i & 0 & 0\\
	0 & 0 & -i & 1\\
	0 & 0 & 0 & -i
\end{pmatrix}\quad\quad\quad\quad
\begin{pmatrix}
	1 & 0 & 0 & 0\\
	0 & 1 & 0 & 0\\
	0 & 0 & -1 & 0\\
	0 & 0 & 0 & -1
\end{pmatrix}\textnormal{and}
\begin{pmatrix}
	i & 0 & 0 & 0\\
	0 & i & 0 & 0\\
	0 & 0 & -i & 0\\
	0 & 0 & 0 & -i
\end{pmatrix}
$$

$$
\begin{pmatrix}
	\frac{-1+\sqrt{3}i}{2} & 0 & 0 & 0\\
	0 & \frac{-1-\sqrt{3}i}{2} & 0 & 0\\
	0 & 0 & \frac{1+\sqrt{3}i}{2} & 0\\
	0 & 0 & 0 & \frac{1-\sqrt{3}i}{2}
\end{pmatrix}\textnormal{and}
\begin{pmatrix}
	\frac{\sqrt{3}+i}{2} & 0 & 0 & 0\\
	0 & \frac{-\sqrt{3}+i}{2} & 0 & 0\\
	0 & 0 & \frac{\sqrt{3}-i}{2} & 0\\
	0 & 0 & 0 & \frac{-\sqrt{3}-i}{2}
\end{pmatrix}$$

$$
\begin{pmatrix}
	\frac{\sqrt{2}+\sqrt{2}i}{2} & 0 & 0 & 0\\
	0 & \frac{-\sqrt{2}+\sqrt{2}i}{2} & 0 & 0\\
	0 & 0 & \frac{\sqrt{2}-\sqrt{2}i}{2} & 0\\
	0 & 0 & 0 & \frac{-\sqrt{2}-\sqrt{2}i}{2}
\end{pmatrix}
$$
where we have grouped the Jordan forms of $M_0$ and $i\cdot M_0$ together. Thus for any geometric operator $\mathcal{P}$ whose local monodromies at $0$ and $\infty$ have Jordan forms from the above list, the shift $\mathcal{P}(\frac{1}{4})$ is locally geometric.

To find such an operator, we start with a family of octic arrangements defined by
\begin{align*}
xyzw(x + y) (y + z) (x + y + z-w) (x + ty + z-w)=0
\end{align*}
This is the Arrangement 5 from \cite{Cy-Ko}. This family admits a symmetry of order 2 and taking the quotient by this symmetry yields a new family of octic arrangements. The Picard-Fuchs operator of the original family is a pull-back of the Picard-Fuchs operator of the quotient family. The latter has one singularity with appropriate Jordan form. To get a family with two such singularities, located precisely at $0$ and $\infty$, we make the substitution $t\mapsto\frac{(t+1)^2}{4t}$.

Explicitly, $X_t\xrightarrow{2:1}\mathbb{P}^3$ is the double cover branched along the octic arrangement
\begin{align*}
yw \left(x -y \right) \left(\frac{\left(t +1\right)^{2} }{4 t}x^{2}-w^{2}-2 z w -z^{2}\right) \left(\frac{\left(t +1\right)^{2} }{4 t}x^{2}-w^{2}+2 z w -z^{2}\right) \left(\frac{\left(t +1\right)^{2} }{4 t}x-\frac{\left(t +1\right)^{2} }{4 t}y+w \right)
\end{align*}
Resolutions of singularities $\widetilde{X}_t$ form a one-parameter family of Calabi-Yau threefolds over $\mathbb{P}^1\setminus\{0,1,-1,\infty\}$. The Picard-Fuchs operator of this family is
\begin{align*}
\mathcal{P}_{loc}&=256 \left(t -1\right)^{3} \left(t +1\right)^{3} \Theta^{4}+512 \left(t^{2}-4 t +1\right) \left(t -1\right)^{2} \left(t +1\right)^{2} \Theta^{3}\\
&+32 \left(t -1\right) \left(t +1\right) \left(11 t^{4}-32 t^{3}+138 t^{2}-32 t +11\right) \Theta^{2}\\
&+32\left(3 t^{6}-12 t^{5}-35 t^{4}-104 t^{3}-35 t^{2}-12 t +3\right) \Theta +9 \left(t +1\right) \left(t -1\right)^{5}
\end{align*}
with Riemann symbol (showing singularities and their corresponding local exponents):
\[\left\{\begin{tabular}{*{4}c}
	0& $1$& --1&$\infty$\\ 
	\hline
	1/4& 0& 0& --1/4\\
	1/4& 2& 0& --1/4\\
	3/4& 2& 0& --3/4\\
	3/4& 4& 0& --3/4\\
\end{tabular}\right\}\]
Using Frobenius bases one finds that the Jordan form of both $M_0$ and $M_\infty$ is
$$
\begin{pmatrix}
	i & 1 & 0 & 0\\
	0 & i & 0 & 0\\
	0 & 0 & -i & 1\\
	0 & 0 & 0 & -i
\end{pmatrix}
$$
Furthermore, all local exponents at $-1$ are equal, so it is a MUM point. From Corollary \ref{cor:mum} we obtain:
\begin{theorem}
Let $k\in\mathbb{Z}\setminus 2\mathbb{Z}$. The operator $\mathcal{P}_{loc}(\frac{k}{4})$ is locally geometric but not geometric.
\end{theorem}

\subsection{Quadratic twists and thin monodromy}\label{ssec:A/B}

Recall that if a geometric operator satisfies Assumption A or B from \cite{simion}, its monodromy group is of infinite index (see Section \ref{ssec:PFsymplectic}). In this section we explain how one can construct plethora of examples of thin monodromy groups satisfying neither Assumption A, nor \mbox{Assumption B.} The construction utilizes shifts of local exponents induced by quadratic twists of one-parameter families of double octics.

A family of octic arrangements $O_t$ with defining polynomials $P_t(x,y,z,w)$ does not uniquely determine the family of double covers $X_t\xrightarrow{2:1}\mathbb{P}^3$ branched along $O_t$. Indeed, for any rational function $\phi(t)$ we can form the \emph{twisted family} $X^{\phi(t)}_t$ defined by $u^2=\phi(t)^{-1}\cdot P_t(x,y,z,w)$. Note that for generic $t\in\mathbb{P}^1$ the fibers $X_t$ and $X^{\phi(t)}_t$ are isomorphic. Nevertheless, the families themselves need not be, e.g. the twisted family can have more singularities.

From the explicit formula (\ref{period-function}) for the period function $y(t)$ of a double octic family $\widetilde{X}_t$, we immediately see that the period function of the twisted family $\widetilde{X}^{\phi(t)}_t$ is given by $\sqrt{\phi(t)}y(t)$. Consider the case $\phi(t)=t^k$ for some $k\in\mathbb{Z}$. Let $\mathcal{P}$ be the Picard-Fuchs operator of the original family. Since $\mathcal{P}y=0$ implies $\mathcal{P}(\tfrac{k}{2})\sqrt{t^k}y=0$, the Picard-Fuchs operator of the twisted family is $\mathcal{P}(\tfrac{k}{2})$. In particular, half-integral shifts of double octic operators are geometric. 

Assume that $\mathcal{P}$ is a double octic operator satisfying Assumption A or Assumption B. The monodromy group $Mon(\mathcal{P})$ is thin by Lemma \ref{lem:AB}, so the monodromy group $Mon\left(\mathcal{P}(\frac{k}{2})\right)$ is also thin by Theorem \ref{th:1/2}. Assume that $0$ is a non-singular point of $\mathcal{P}$, which can always be arranged using a change of variables. Local exponents at a non-singular point are $0,1,2,3$. Thus local exponents at $0$ of the shifted operator $\mathcal{P}(\frac{k}{2})$ are $\frac{k}{2},\frac{k+2}{2},\frac{k+4}{2},\frac{k+6}{2}$. If $k\not\in 2\mathbb{Z}$, after the shift the order of the local monodromy at $0$ is $N=2$. By Lemma \ref{lem:AB} neither Assumption A, nor Assumption B is satisfied at $0$.
\begin{theorem}\label{cor:A/B}
Let $\mathcal{P}$ be a double octic operator satisfying Assumption A or Assumption B. Assume that $0$ is a non-singular point and let $k\in\mathbb{Z}\setminus 2\mathbb{Z}$. Then the operator $\mathcal{P}(\frac{k}{2})$ is geometric with thin monodromy group but satisfies neither Assumption A, nor Assumption B.
\end{theorem}

We end this paper by listing several examples of double octic operators satisfying Assumption A or Assumption B. The list contains the octic arrangement defining the family, the corresponding Picard-Fuchs operator and its Riemann symbol. The numbering follows that of \cite{Cy-Ko}. We draw special attention to operators satisfying Assumption B, since no example of such operator appears in \cite{simion}.

\begin{footnotesize}
\begin{center}
	\Large{Double octic operators satisfying Assumption A}
\end{center}

\textbf{Operator No. 2}

\( xyzw \left( z+w \right)  \left( z+y
\right)  \left( y+x \right)  \left( x+wt \right) \) 

\({\Theta}^{4}
-t \left(\Theta+\tfrac12 \right) ^{4}\)

\[\left\{\begin{tabular}{*{3}c}
	0& $1$& $\infty$\\ 
	\hline
	0&  0& 1/2\\
	0&  1& 1/2\\
	0&  1& 1/2\\
	0&  2& 1/2\\
\end{tabular}\right\}\]

\textbf{Operator No. 10}

\(xyzw \left( z-w \right)  \left( z+y \right)  \left( y+x \right)  \left( -x+yt+zt+w \right) \)

\(\Theta\, \left( \Theta-1 \right)  \left(\Theta-\frac12 \right) ^{2}
+\frac12\,t{\Theta}^{2} \left( 4\,{\Theta}^{2}+1 \right) 
+\frac1{16}\,{t}^{2} \left( 2\,\Theta+1 \right) ^{4}
\)

\[\left\{\begin{tabular}{*{4}c}
	0& --1&$\infty$\\ 
	\hline
	0& 0& 1/2\\
	1/2& 1/2& 1/2\\
	1/2& 1/2& 1/2\\
	1& 1& 1/2\\
\end{tabular}\right\}\]

\textbf{Operator No. 16}

\(xyzw \left( z+y \right)  \left( y+x \right)  \left( -y+w+zt \right)
\left( xt-y+w \right) 
\)

\({\Theta}^{4}
+\frac14\,t \left( 2\,{\Theta}^{2}+2\,\Theta+1 \right)  \left( 2\,\Theta+1 \right) ^{2}
+\frac14\,{t}^{2} \left( 2\,\Theta+3 \right)  \left( 2\,\Theta+1 \right)  \left( \Theta+1 \right) ^{2}
\)

\[\left\{\begin{tabular}{*{4}c}
	0& --1&$\infty$\\ 
	\hline
	0& 0& 1/2\\
	0& 1/2& 1\\
	0& 1/2& 1\\
	0& 1& 3/2\\
\end{tabular}\right\}\]

\textbf{Operator No. 242}

\(xyzw \left( x+z-w \right)  \left( x+z+y \right)  \left( x+xt+ty+y+tw \right)  \left( x+tw+ty+y-zt \right) \)

\({\Theta}^{4}
+t \left( 5\,{\Theta}^{4}+10\,{\Theta}^{3}+10\,{\Theta}^{2}+5\,\Theta+1 \right)
+{t}^{2} \left( 9\,{\Theta}^{2}+18\,\Theta+14 \right)  \left( \Theta+1
\right) ^{2}\\
+{t}^{3} \left( \Theta+2 \right)  \left( \Theta+1 \right)  \left( 7\,{\Theta}^{2}+21\,\Theta+18 \right)
+2\,{t}^{4} \left( \Theta+3 \right)  \left( \Theta+1 \right)  \left( \Theta+2 \right) ^{2}
\)

\[\left\{\begin{tabular}{*{4}c}
	0& --1/2& --1&$\infty$\\ 
	\hline
	0& 0& 0& 1\\
	0& 1& 0& 2\\
	0& 1& 0& 2\\
	0& 2& 0& 3\\
\end{tabular}\right\}\]

\textbf{Operator No. 246}

\(xyzw \left( x+z+y \right)  \left( w+zt-x \right)  \left( x-w+ty+y+z \right)  \left( x+tw+ty+y-zt \right) 
\)

\(\Theta\, \left( \Theta-1 \right)  \left(\Theta -\frac12\right) ^{2}
+\frac14\,t{\Theta}^{2} \left( 20\,{\Theta}^{2}+11 \right) 
+{t}^{2} \left( 9\,{\Theta}^{4}+18\,{\Theta}^{3}+22\,{\Theta}^{2}+13\,\Theta+{\frac {51}{16}} \right)\\ 
+\frac1{16}\,{t}^{3} \left( 4\,{\Theta}^{2}+8\,\Theta+7 \right)  \left( 28\,{\Theta}^{2}+56\,\Theta+29 \right) 
+\frac18\,{t}^{4} \left( 2\,\Theta+3 \right) ^{4}
\)

\[\left\{\begin{tabular}{*{4}c}
	0& --1/2& --1&$\infty$\\ 
	\hline
	0& 0& --1/2& 3/2\\
	1/2& 1& 0& 3/2\\
	1/2& 1& 0& 3/2\\
	1& 2& 1/2& 3/2\\
\end{tabular}\right\}\]

\medskip

\begin{center}
	\Large{Double octic operators satisfying Assumption B}
\end{center}

\textbf{Operator No. 21}

\(xyzw \left( z+y \right)  \left( y+x \right)  \left( x+zt-w \right)  \left( yt+wt-x+w \right) 
\)

\({\Theta}^{4}
+t \left( 7\,{\Theta}^{4}+2\,{\Theta}^{3}+3\,{\Theta}^{2}+2\,\Theta+{\frac {7}{16}} \right)
+{t}^{2} \left( 19\,{\Theta}^{4}+16\,{\Theta}^{3}+21\,{\Theta}^{2}+10\,\Theta+\frac{7}{4} \right)\\
+{t}^{3} \left( 25\,{\Theta}^{4}+42\,{\Theta}^{3}+50\,{\Theta}^{2}+27\,\Theta+{\frac {23}{4}} \right)
+{t}^{4} \left( 16\,{\Theta}^{4}+44\,{\Theta}^{3}+56\,{\Theta}^{2}+34\,\Theta+8 \right) 
+4\,{t}^{5} \left( \Theta+1 \right) ^{4}
\)

\[\left\{\begin{tabular}{*{4}c}
	0& --1/2& --1& $\infty$\\ 
	\hline
	0&  0& 0&1\\
	0&  1& 0&1\\
	0&  3& 0&1\\
	0&  4& 0&1\\
\end{tabular}\right\}\]

\textbf{Operator No. 53}

\(xyzw \left( w+z \right)  \left( y+x \right)  \left( xt+w+yt-zt \right)  \left( -x+w+yt+z \right) 
\)

\({\Theta}^{4}
+t \left( {\Theta}^{4}+8\,{\Theta}^{3}+7\,{\Theta}^{2}+3\,\Theta+{\frac {9}{16}} \right) 
+{t}^{2} \left( -2\,{\Theta}^{4}+4\,{\Theta}^{3}+19\,{\Theta}^{2}+15\,\Theta+{\frac {71}{16}} \right)\\
+{t}^{3} \left( -2\,{\Theta}^{4}-12\,{\Theta}^{3}-5\,{\Theta}^{2}+3\,\Theta+{\frac {39}{16}} \right)
+{t}^{4} \left( {\Theta}^{4}-4\,{\Theta}^{3}-11\,{\Theta}^{2}-9\,\Theta-{\frac {39}{16}} \right) 
+{t}^{5} \left( \Theta+1 \right) ^{4}
\)

\[\left\{\begin{tabular}{*{4}c}
	0& 1& --1& $\infty$\\ 
	\hline
	0&  0& 0&1\\
	0&  1& 0&1\\
	0&  3& 0&1\\
	0&  4& 0&1\\
\end{tabular}\right\}\]

\textbf{Operator No. 96}

\(xyzw \left( x+y \right)  \left( x+y-z+w \right)  \left( y+w+zt \right)  \left( -x-w+ty-zt \right) 
\)

\(
{\Theta}^{2} \left( \Theta-2 \right) ^{2}
+t \left( \frac92\,{\Theta}^{4}-11\,{\Theta}^{3}+7\,{\Theta}^{2}+\frac12\,\Theta+\frac12 \right) 
+{t}^{2} \left( {\frac {33}{4}}\,{\Theta}^{4}-8\,{\Theta}^{3}+{\frac {73}{16}}\,{\Theta}^{2}+\frac12\,\Theta+\frac34 \right) \\
+{t}^{3} \left( {\frac {63}{8}}\,{\Theta}^{4}+{\frac {13}{4}}\,{\Theta}^{3}+{\frac {189}{32}}\,{\Theta}^{2}+{\frac {71}{32}}\,\Theta+{\frac {25}{32}} \right) 
+{t}^{4} \left( {\frac {33}{8}}\,{\Theta}^{4}+7\,{\Theta}^{3}+{\frac {31}{4}}\,{\Theta}^{2}+{\frac {33}{8}}\,\Theta+{\frac {63}{64}} \right) \\
+{t}^{5} \left( {\frac {9}{8}}\,{\Theta}^{4}+{\frac {13}{4}}\,{\Theta}^{3}+{\frac {133}{32}}\,{\Theta}^{2}+{\frac {81}{32}}\,\Theta+{\frac {77}{128}} \right) 
+\frac18\,{t}^{6} \left( \Theta+1 \right) ^{4}
\)

\[\left\{\begin{tabular}{*{4}c}
	0& --1& --2& $\infty$\\ 
	\hline
	0&  0& 0&1\\
	0&  0& 1/2&1\\
	2&  0& 3/2&1\\
	2&  0& 2&1\\
\end{tabular}\right\}\]

\textbf{Operator No. 100}

\(xyzw \left( x+y-z+w \right)  \left( ty-w-zt \right)  \left( y+w+zt \right)  \left( zt+x+y \right) 
\)

\(
{\Theta}^{4}
+t \left( 9\,{\Theta}^{4}+10\,{\Theta}^{3}+{\frac {37}{4}}\,{\Theta}^{2}+{\frac {17}{4}}\,\Theta+{\frac {13}{16}} \right) 
+{t}^{2} \left( 33\,{\Theta}^{4}+76\,{\Theta}^{3}+92\,{\Theta}^{2}+55\,\Theta+{\frac {111}{8}} \right) \\
+{t}^{3} \left( 63\,{\Theta}^{4}+226\,{\Theta}^{3}+{\frac {1389}{4}}\,{\Theta}^{2}+{\frac {1003}{4}}\,\Theta+{\frac {291}{4}} \right) 
+{t}^{4} \left( 66\,{\Theta}^{4}+328\,{\Theta}^{3}+{\frac {1249}{2}}\,{\Theta}^{2}+525\,\Theta+{\frac {337}{2}} \right) \\
+{t}^{5} \left( 36\,{\Theta}^{4}+232\,{\Theta}^{3}+536\,{\Theta}^{2}+516\,\Theta+180 \right) 
+8\,{t}^{6} \left( \Theta+3 \right) ^{2} \left( \Theta+1 \right) ^{2}
\)

\[\left\{\begin{tabular}{*{4}c}
	0& --1/2& --1&$\infty$\\ 
	\hline
	0&  0& 0&1\\
	0&  1/2& 0&1\\
	0&  3/2& 0&3\\
	0&  2& 0&3\\
\end{tabular}\right\}\]

\textbf{Operator No. 155}

\(xyzw \left( x+y+tz+w \right)  \left( x+tz+y \right)  \left( tz+w+xt \right)  \left( yt-x-y-w-z \right) 
\)

\({\Theta}^{4}
+t \left( -6\,{\Theta}^{4}-6\,{\Theta}^{3}-\frac{13}2\,{\Theta}^{2}-\frac72\,\Theta-\frac34 \right) 
+{t}^{2} \left( 18\,{\Theta}^{4}+30\,{\Theta}^{3}+{\frac {167}{4}}\,{\Theta}^{2}+31\,\Theta+9 \right)\\ 
+{t}^{3} \left( -35\,{\Theta}^{4}-78\,{\Theta}^{3}-{\frac {243}{2}}\,{\Theta}^{2}-{\frac {411}{4}}\,\Theta-{\frac {279}{8}} \right) 
+{t}^{4} \left( 48\,{\Theta}^{4}+132\,{\Theta}^{3}+{\frac {887}{4}}\,{\Theta}^{2}+{\frac {761}{4}}\,\Theta+{\frac {1083}{16}} \right) \\
+{t}^{5} \left( -48\,{\Theta}^{4}-156\,{\Theta}^{3}-{\frac {1103}{4}}\,{\Theta}^{2}-232\,\Theta-{\frac {315}{4}} \right) 
+{t}^{6} \left( 35\,{\Theta}^{4}+132\,{\Theta}^{3}+243\,{\Theta}^{2}+{\frac {831}{4}}\,\Theta+{\frac {1089}{16}} \right) \\
+{t}^{7} \left( -18\,{\Theta}^{4}-78\,{\Theta}^{3}-{\frac {599}{4}}\,{\Theta}^{2}-{\frac {539}{4}}\,\Theta-{\frac {741}{16}} \right) 
+{t}^{8} \left( 6\,{\Theta}^{4}+30\,{\Theta}^{3}+{\frac {121}{2}}\,{\Theta}^{2}+{\frac {113}{2}}\,\Theta+{\frac {81}{4}} \right) 
-\frac1{16}\,{t}^{9} \left( 2\,\Theta+3 \right) ^{4}
\)

\[\left\{\begin{tabular}{*{4}c}
	0& 1& $\frac{1\pm\sqrt{-3}}{2}$& $\infty$\\ 
	\hline
	0&  0& 0&3/2\\
	0&  0& 1/2&3/2\\
	0&  0& 5/2&3/2\\
	0&  0& 3&3/2\\
\end{tabular}\right\}\]

\textbf{Operator No. 200}

\(yxzw \left( x+y+z+w \right)  \left( zt-w-y \right)  \left( -x+t+ty \right)  \left( zt+t-x-y \right) 
\)

\({\Theta}^{4}
+t \left( 6\,{\Theta}^{4}+6\,{\Theta}^{3}+\frac{13}2\,{\Theta}^{2}+\frac72\,\Theta+\frac34 \right) 
+{t}^{2} \left( 18\,{\Theta}^{4}+30\,{\Theta}^{3}+{\frac {167}{4}}\,{\Theta}^{2}+31\,\Theta+9 \right) \\
+{t}^{3} \left( 35\,{\Theta}^{4}+78\,{\Theta}^{3}+{\frac {243}{2}}\,{\Theta}^{2}+{\frac {411}{4}}\,\Theta+{\frac {279}{8}} \right) 
+{t}^{4} \left( 48\,{\Theta}^{4}+132\,{\Theta}^{3}+{\frac {887}{4}}\,{\Theta}^{2}+{\frac {761}{4}}\,\Theta+{\frac {1083}{16}} \right) \\
+{t}^{5} \left( 48\,{\Theta}^{4}+156\,{\Theta}^{3}+{\frac {1103}{4}}\,{\Theta}^{2}+232\,\Theta+{\frac {315}{4}} \right) 
+{t}^{6} \left( 35\,{\Theta}^{4}+132\,{\Theta}^{3}+243\,{\Theta}^{2}+{\frac {831}{4}}\,\Theta+{\frac {1089}{16}} \right) \\
+{t}^{7} \left( 18\,{\Theta}^{4}+78\,{\Theta}^{3}+{\frac {599}{4}}\,{\Theta}^{2}+{\frac {539}{4}}\,\Theta+{\frac {741}{16}} \right) 
+{t}^{8} \left( 6\,{\Theta}^{4}+30\,{\Theta}^{3}+{\frac {121}{2}}\,{\Theta}^{2}+{\frac {113}{2}}\,\Theta+{\frac {81}{4}} \right) 
+\frac1{16}\,{t}^{9} \left( 2\,\Theta+3 \right) ^{4}
\)

\[\left\{\begin{tabular}{*{4}c}
	0& --1& $\frac{-1\pm\sqrt{-3}}{2}$& $\infty$\\ 
	\hline
	0&  0& 0&3/2\\
	0&  0& 1/2&3/2\\
	0&  0& 5/2&3/2\\
	0&  0& 3&3/2\\
\end{tabular}\right\}\]

\textbf{Operator No. 267}

\(xyzw \left( -y+tw+ty-zt \right)  \left( x+y+zt-z \right)  \left( x+ty-z+w \right)  \left( xt+tw+ty-z \right) 
\)

\({\Theta}^{4}
+t \left( -8\,{\Theta}^{4}-7\,{\Theta}^{3}-8\,{\Theta}^{2}-\frac92\,\Theta-1 \right) 
+{t}^{2} \left( {\frac {105}{4}}\,{\Theta}^{4}+{\frac {129}{2}}\,{\Theta}^{3}+90\,{\Theta}^{2}+{\frac {249}{4}}\,\Theta+{\frac {303}{16}} \right) \\
+{t}^{3} \left( -{\frac {175}{4}}\,{\Theta}^{4}-{\frac {471}{2}}\,{\Theta}^{3}-{\frac {1711}{4}}\,{\Theta}^{2}-{\frac {1515}{4}}\,\Theta-{\frac {2283}{16}} \right) 
+{t}^{4} \left( {\frac {119}{4}}\,{\Theta}^{4}+{\frac {863}{2}}\,{\Theta}^{3}+{\frac {4473}{4}}\,{\Theta}^{2}+{\frac {4885}{4}}\,\Theta+{\frac {34559}{64}} \right) \\
+{t}^{5} \left( {\frac {105}{4}}\,{\Theta}^{4}-{\frac {717}{2}}\,{\Theta}^{3}-{\frac {3303}{2}}\,{\Theta}^{2}-{\frac {9057}{4}}\,\Theta-{\frac {18741}{16}} \right) 
+{t}^{6} \left( -{\frac {165}{2}}\,{\Theta}^{4}-99\,{\Theta}^{3}+{\frac {2313}{2}}\,{\Theta}^{2}+{\frac {4653}{2}}\,\Theta+{\frac {94485}{64}} \right) \\
+{t}^{7} \left( {\frac {165}{2}}\,{\Theta}^{4}+561\,{\Theta}^{3}+{\frac {459}{2}}\,{\Theta}^{2}-{\frac {1695}{2}}\,\Theta-{\frac {58965}{64}} \right) 
+{t}^{8} \left( -{\frac {105}{4}}\,{\Theta}^{4}-{\frac {1137}{2}}\,{\Theta}^{3}-{\frac {2259}{2}}\,{\Theta}^{2}-{\frac {3201}{4}}\,\Theta-{\frac {627}{16}} \right) \\
+{t}^{9} \left( -{\frac {119}{4}}\,{\Theta}^{4}+{\frac {387}{2}}\,{\Theta}^{3}+{\frac {3027}{4}}\,{\Theta}^{2}+{\frac {3897}{4}}\,\Theta+{\frac {25953}{64}} \right) 
+{t}^{10} \left( {\frac {175}{4}}\,{\Theta}^{4}+{\frac {229}{2}}\,{\Theta}^{3}+{\frac {259}{4}}\,{\Theta}^{2}-{\frac {375}{4}}\,\Theta-{\frac {1405}{16}} \right) \\
+{t}^{11} \left( -{\frac {105}{4}}\,{\Theta}^{4}-{\frac {291}{2}}\,{\Theta}^{3}-333\,{\Theta}^{2}-{\frac {1455}{4}}\,\Theta-{\frac {2535}{16}} \right) 
+{t}^{12} \left( 8\,{\Theta}^{4}+57\,{\Theta}^{3}+158\,{\Theta}^{2}+{\frac {399}{2}}\,\Theta+96 \right) 
-{t}^{13} \left( \Theta+2 \right) ^{4}
\)

\[\left\{\begin{tabular}{*{7}c}
	0& 1/2& 1& 2& --1& $\frac{1\pm\sqrt{-3}}{2}$& $\infty$\\ 
	\hline
	0&  0& 0& 0& 0& 0& 2 \\
	0&  1& 0& 1& 1& 0& 2 \\
	0&  3& 0& 3& 3& 1& 2 \\
	0&  4& 0& 4& 4& 1& 2 \\
\end{tabular}\right\}\]

\textbf{Operator No. 275}

\(xyz \left( 4\,z+w-4\,zt-tw+8\,ty \right)  \left( 2\,z+2\,zt+t+2\,xt
\right)  \left( 4\,y-w+4\,ty-t-2\,x+2\,xt \right)  \left( xt+ty-zt+y+z
\right)  
\)

\({\Theta}^{4}
+{t}^{2} \left( -15\,{\Theta}^{4}+48\,{\Theta}^{3}+65\,{\Theta}^{2}+34\,\Theta+{\frac {39}{4}} \right) 
+{t}^{4} \left( 21\,{\Theta}^{4}-480\,{\Theta}^{3}+78\,{\Theta}^{2}-114\,\Theta-123 \right) \\
+{t}^{6} \left( 341\,{\Theta}^{4}+528\,{\Theta}^{3}+1627\,{\Theta}^{2}+2364\,\Theta+{\frac {2475}{2}} \right) 
+{t}^{8} \left( -285\,{\Theta}^{4}+192\,{\Theta}^{3}+948\,{\Theta}^{2}+1980\,\Theta+1620 \right) \\
+{t}^{10} \left( -1197\,{\Theta}^{4}-11952\,{\Theta}^{3}-45081\,{\Theta}^{2}-77094\,\Theta-{\frac {203013}{4}} \right) 
+81\,{t}^{12} \left( \Theta+3 \right)  \left( 23\,{\Theta}^{3}+219\,{\Theta}^{2}+685\,\Theta+687 \right) \\
-729\,{t}^{14} \left( \Theta+5 \right)  \left( \Theta+3 \right)  \left( \Theta+2 \right)  \left( \Theta+6 \right) 
\)

\[\left\{\begin{tabular}{*{7}c}
	0& 1/3& 1& --1/3& --1& $\pm\frac{\sqrt{-3}}{3}$& $\infty$\\ 
	\hline
	0&  0& 0& 0& 0& 0& 2 \\
	0&  1& 0& 1& 0& 0& 3 \\
	0&  3& 0& 3& 0& 1& 5 \\
	0&  4& 0& 4& 0& 1& 6 \\
\end{tabular}\right\}\]

\textbf{Operator No. 276}

\(xyzw   \left( -x+x{t}-{t}y-z \right)  \left( -y-z-w+{t} \right)  \left( x+z+w \right)  \left( x+{t}y+{t} \right)\)

\(
{\Theta}^{4}
+t\left(  -13\,{\Theta}^{4}+10\,{\Theta}^{3}+11/2\,{\Theta}^{2}+1/2\,\Theta-1/16
\right) 
+t^{2}\left(
48\,{\Theta}^{4}-72\,{\Theta}^{3}+9\,{\Theta}^{2}-9\,\Theta-6
\right) \\
+ t^{3}\left( -24\,{\Theta}^{4
}+120\,{\Theta}^{3}+117\,{\Theta}^{2}+114\,\Theta+{\frac {145}{4}} \right)+ 
t^{4}\left( -86\,{\Theta}^{4}-280\,{\Theta}^{3}-459\,{\Theta}^{2}-379\,\Theta-120
\right) \\
+ t^{5}\left( 86\,{\Theta}^{4}+236\,{\Theta}^{3}+360\,{\Theta}^{2}+269\,\Theta
+{\frac {597}{8}} \right) + 
t^{6}\left( 24\,{\Theta}^{4}+264\,{\Theta}^{3}+747\,{\Theta}^{2}+897\,\Theta+398 \right)  \\
+ t^{7}\left(
-48\,{\Theta}^{4}-360\,{\Theta}^{3}-981\,{\Theta}^{2}-1170\,\Theta-{\frac
	{2055}{4}} \right) 
+ t^{8}\left( 13\,{\Theta}^{4}+88\,{\Theta}^{3
}+215\,{\Theta}^{2}+227\,\Theta+88 \right)  \\
+t^{9}\left( -{\Theta}^{4}-6\,{\Theta}^{3}-{\frac {27}{2}}\,{\Theta}^{2}-{\frac {27}{2}}\,\Theta-{\frac {81}{16}} \right)
\)

\[\left\{\begin{tabular}{*{5}c}
	0& 1& --1& $3\pm2\sqrt{2}$& $\infty$\\ 
	\hline
	0&  0& 0& 0& 3/2\\
	0&  0& 0& 1& 3/2\\
	0&  0& 1& 3& 3/2\\
	0&  0& 1& 4& 3/2\\
\end{tabular}\right\}\]
\end{footnotesize}

\end{document}